\documentclass{amsart}
\usepackage[latin1]{inputenc}
\usepackage{fancyhdr}
\usepackage{ucs}
\usepackage{amsmath}
\usepackage{amsthm}
\usepackage{amsfonts}
\usepackage{amssymb}
\usepackage[dvips]{graphicx}

\newcommand{\abs}[1]{\left\vert#1\right\vert}
\newcommand{\mathsym}[1]{{}}

\newtheorem{theorem}{Theorem}
\theoremstyle{plain}

\newtheorem{example}{Example}

\newtheorem{lemma}{Lemma}

\newtheorem{remark}{Remark}

\numberwithin{equation}{section}
\newcommand{\R}{\mathbb{R}}

\begin{document}
\title{Optimal transport and Tessellation}
\author{Martin Huesmann}

\begin{abstract}
Optimal transport from the volume measure to a convex combination of Dirac measures yields a tessellation of a Riemannian manifold into pieces of arbitrary relative size. This tessellation is studied for the cost functions $c_p(z,y)=\frac{1}{p}d^p(z,y)$ and $1\leq p<\infty$. Geometric descriptions of the tessellations for all $p$ is obtained for compact subsets of the Euclidean space. For $p=2$ this approach yields Laguerre tessellations. For $p=1$ it induces Johnson Mehl diagrams for all compact Riemannian manifolds.
\end{abstract}

\keywords{Optimal transportation, Laguerre tessellation}
\maketitle

\section{Introduction}
Given two probability measures $\mu$ and $\nu$ on some Riemannian manifold $M$, a transportation map from $\mu$ to $\nu$ is a Borel measurable map $\tilde T:M\to M$ pushing $\mu$ forward to $\nu$, i.e. $\tilde T\#\mu=\nu$. In the theory of optimal transportation one is interested in a special transportation map, namely the one minimising
\[
\int_M c(x,\tilde T(x)) d\mu(x)
\]
the transportation cost among all possible transportation maps, for a given cost function $c:M\times M\to\mathbb{R}$. This minimiser, if it exists, is called optimal transportation map. \\

A Voronoi tessellation of a Riemannian manifold $M$ with respect to a discrete set of points $S$ is a decomposition of $M$ into different cells with centre $s\in S$. Each cell consists of those points which are closer to its centre than to any other centre. A rather natural extension are the Laguerre tessellations where different centres can have different weights, see \cite{lautensack2007}.\\

The theory of optimal transport is a very natural setting for studying such tessellations. Just take $\mu$ as the normalised volume measure, $\nu$ as a convex combination of Dirac-measures and as the cost function $c(x,y)=\frac{1}{2}d^2(x,y)$, the geodesic distance squared. The resulting transportation map is a Laguerre tessellation. In this framework the generalisation of Laguerre tessellations to other cost functions like $c(x,y)=\frac{1}{p} d^p(x,y)$ is straightforward (at least to write the tessellation down).\\

Some of the phenomena studied in this paper already appeared in the work of Ma, Trudinger and Wang, Loeper, Kim and McCann (see \cite{ma2005,loeper2009regularity,kim-2007}), where they studied the regularity of the optimal transport map. In the case above for $p=2$ and the manifold $M$ being the hyperbolic disc, the resulting "cells" do not have to be connected any more. If one smears the Dirac points slightly one ends up with a measure which is absolutely continuous to the volume measure (the density can even be very nice) but with a discontinuous transportation map! Therefore, even the transport between two measures being absolutely continuous to the volume measure can be rather irregular. However, in the case $p=2$ the Ma-Trudinger-Wang condition is known to hold on the sphere. In particular, this implies that all cells have to be connected. For $p\neq 2$ the condition is not verified. 

The Euclidean case, section \ref{Euclidean case}, was already mentioned in the pioneering work of Gangbo and McCann \cite{GangboMcCann1996} and  R\"uschendorf and Uckelmann \cite{RueUckelmann1997}. They considered this setting as an example for an explicit solution to the transportation problem (example 1.6 in \cite{GangboMcCann1996} and section 1.2 in \cite{RueUckelmann1997}). In this work we are more interested in the geometric and topological properties of the resulting tessellations.

\subsection{General setting}

In \cite{mccann01polar} McCann showed that on a compact connected complete Riemannian manifold $(M,g)$ without boundary there is always an optimal transport map $T$ (defined a.e.) pushing some measure $\mu$, absolutely continuous to the volume measure $vol$, forward to an arbitrary Borel measure $\nu$ on $M$ under the condition that the cost function $c$ is superdifferentiable. In the case $c(x,y)=\frac{1}{2}d^2(x,y)$ the map $T$ is the exponential of minus the gradient of a c-concave function $\Phi$. A function is c-concave iff $(\Phi^c)^c=\Phi$, where 
\[
\Phi^c(x)=\inf_{y\in M} [c(x,y)-\Phi(y)] .
\]
The function $\Phi$, also called potential, arises from the dual formulation of the transportation problem. In the dual formulation one tries to maximise 
\[
 \int \Phi d\mu + \int \Phi^c d\nu.
\]

In the case of $\mu$-a.e. equality, $\Phi$ induces the optimal transportation map $T$.

We consider the special case of $\mu(\cdot)= vol(\cdot)/vol(M)$, the normalised volume measure, and $\nu(\cdot)=\sum_{i\geq 0} \lambda_i \delta_{x_i}(\cdot)$, with $\lambda_i\geq 0$, $\sum \lambda_i=1$ and $x_i\in M$ for all $i$, a (possibly infinite) convex combination of Dirac measures, and the cost function $c_p(x,y)=\frac{1}{p} d^p(x,y)$ for $1\leq p< \infty$, where $d(x,y)$ denotes the geodesic distance between $x$ and $y$. Then, the optimal transport map can be written down explicitly, i.e. the c-concave function $\Phi$ can be written down explicitly (the apparent $p$-dependence is suppressed in the notation for the optimal transport map):
\begin{equation}
 \Phi(z)=\max_{1\leq i\leq n} \Phi_i(z), \quad \text{with} \quad \Phi_i(z)= - \frac{d^p}{p}(x_i,z) + b_i, \quad b_i \in \mathbb{R}\quad\forall i.
\end{equation}
and the optimal transportation map is given by (for $\Phi(z)=\Phi_i(z)$, see \cite{mccann01polar})
\[
 T(z)= \exp\left(\frac{d(z,x_i)}{\abs{\nabla \Phi_i(z)}}\nabla \Phi_i(z)\right)
\]

In particular, this shows that there is always a solution to the (even very much generalised) tessellation problem.

We are interested in the geometry of the mass being allocated to a certain Dirac point. More specific, how does the set $T^{-1}(x_i)$, the cell allocated to $x_i$, look like? Does it have to be connected?\\
A nice way to look at these sets is to interpret them as intersections of ``halfspaces'', i.e.:
\begin{eqnarray*}
 T^{-1}(x_i) &=& \{ q\in M : \Phi_i(q)> \Phi_j(q) \forall j\neq i\} \\
&=& \bigcap_{j\neq i} \{q\in M : \Phi_i(q)>\Phi_j(q)\} \\
&=:& \bigcap_{j\neq i} H^i_j.
\end{eqnarray*}
Thus, a good way to study the geometry of the cells $T^{-1}(x_i)$ is to look for properties invariant under sections of these ``halfspaces''. In this way we can reduce the problem from a many point problem to a two point problem. 

This also directly settles the question of the regularity of the boundary of the cells $T^{-1}(x_i)$ in the case of finitely many Dirac points. To study this regularity, it suffices to study the regularity of the boundary of the halfspaces $\partial H^i_j=\{p\in M: \Phi_i(q)=\Phi_j(q)\}$. This question is answered by an implicit function theorem for subdifferentiable functions, e.g. like in the second appendix to chapter 10 in \cite{villani2-2009}, and yields
\begin{lemma}
$\partial H^i_j$ is (locally) a $(n-1)$- dimensional Lipschitz graph. For all $q\in \partial H^i_j$ with $q\notin \{cutlocus(x_i)\}\cup\{cutlocus(x_j)\}$ there is a neighbourhood $U$ of $q$ such that $U\cap \partial H^i_j$ is the $(n-1)$ dimensional graph of a function which is as smooth as the cost function.
\end{lemma}
Thus, the boundary of the set $T^{-1}(x_i)$ is, as an intersection of sets with locally Lipschitz boundary, itself a set with locally Lipschitz boundary up to the points of intersection with the different ``halfspaces''.
In particular, $\partial H^i_j$ is a $\mu$-null set. Considering just finitely many Dirac points, we can thus define the cells $T^{-1}(x_i)$ to be open.
In the case of infinitely many Dirac points, things become less regular. For example, take a dense countable subset of the unit disc as Dirac points. As the different cells are starlike, see section \ref{general manifolds}, they become stars consisting of many lines.\\

We now turn to the question of the geometry of the cells, and especially the connectedness. First, we will study the case of compact subsets of $\mathbb{R}^n$. Here we distinguish between convex and just simply connected subsets because they behave drastically different. From the point of view of optimal transportation it would be more natural to consider c-convex sets instead of convex sets. However, from the point of view of tessellations the convex setting is more natural. Thus, we will stick to convex and just simply connected subsets of $\R^d$. In section \ref{general manifolds} we deal with the case $p=1$ which can be treated for all Riemannian manifolds yielding an Johnson Mehl diagram. This is in contrast to the case of quadratic cost as mentioned earlier.

\section{Euclidean case} \label{Euclidean case}
\subsection{Convex subsets of $\mathbb{R}^n$}

In this section, we will look at compact simply connected subsets $\Omega$ of $\mathbb{R}^n$. Let us assume until explicitly said otherwise that $\Omega$ is convex. In the case $p=2$ and finitely many Dirac points it was shown in section 4 of \cite{sturm-2009} and also in \cite{142747} that the cells $T^{-1}(x_i)$ are convex polytopes and therefore simply connected. This can be easily seen from the characterisation of this set as an intersection of halfspaces. Indeed, $\partial H^i_j$ is just a straight line. Thus, both $H^i_j$ and $H^j_i$ are convex polytopes. As intersections of convex sets are convex sets, $\bigcap_{j\neq i} H^i_j$ is still a convex polytope and we are done. However, this is the only case where we actually have connected sets. For all other cases of $p$, the sets $T^{-1}(x_i)$ do not have to be connected! Indeed, we have the following
\begin{figure}\label{kleinesp}
\begin{minipage}{13cm}
\begin{minipage}{6cm}
 \includegraphics[scale=0.21]{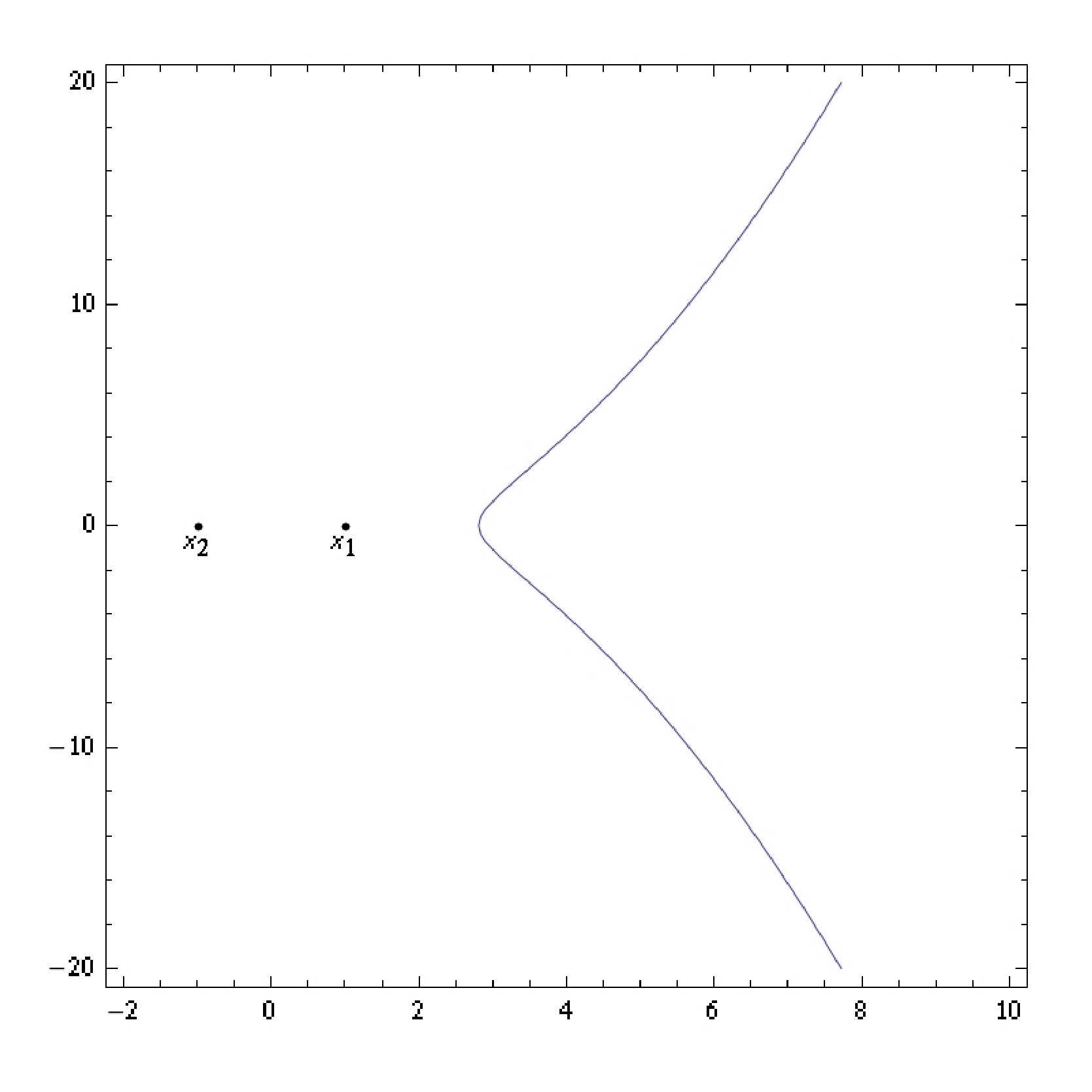}
\end{minipage}
$\dashrightarrow$
\begin{minipage}{6cm}
 \includegraphics[scale=0.21]{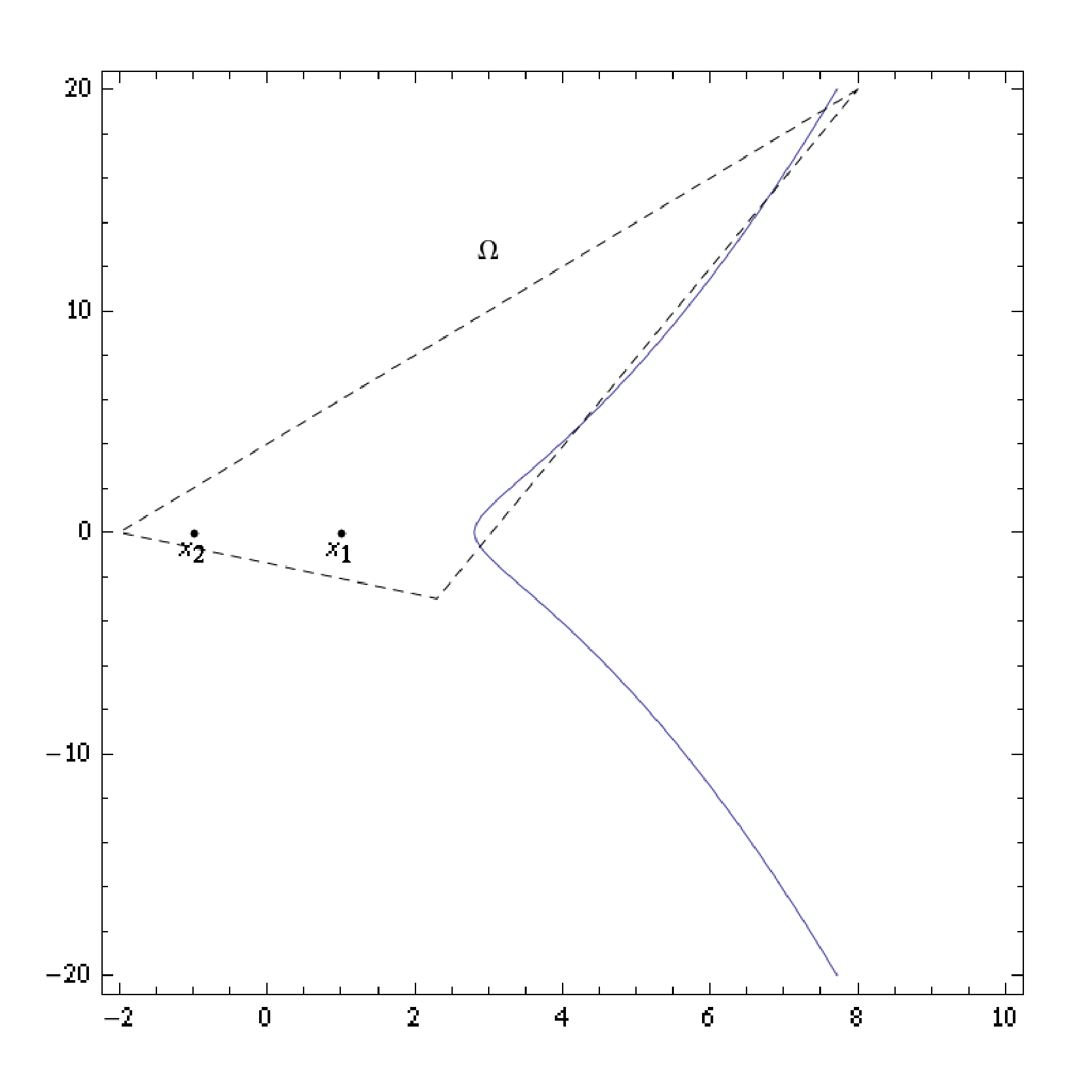}
\end{minipage}
\end{minipage} 
\caption{$p=3/2$, $\alpha=5$}
\end{figure}

\begin{figure}\label{grossesp}
\begin{minipage}{13cm}
\begin{minipage}[c]{6cm}
 \includegraphics[scale=0.21]{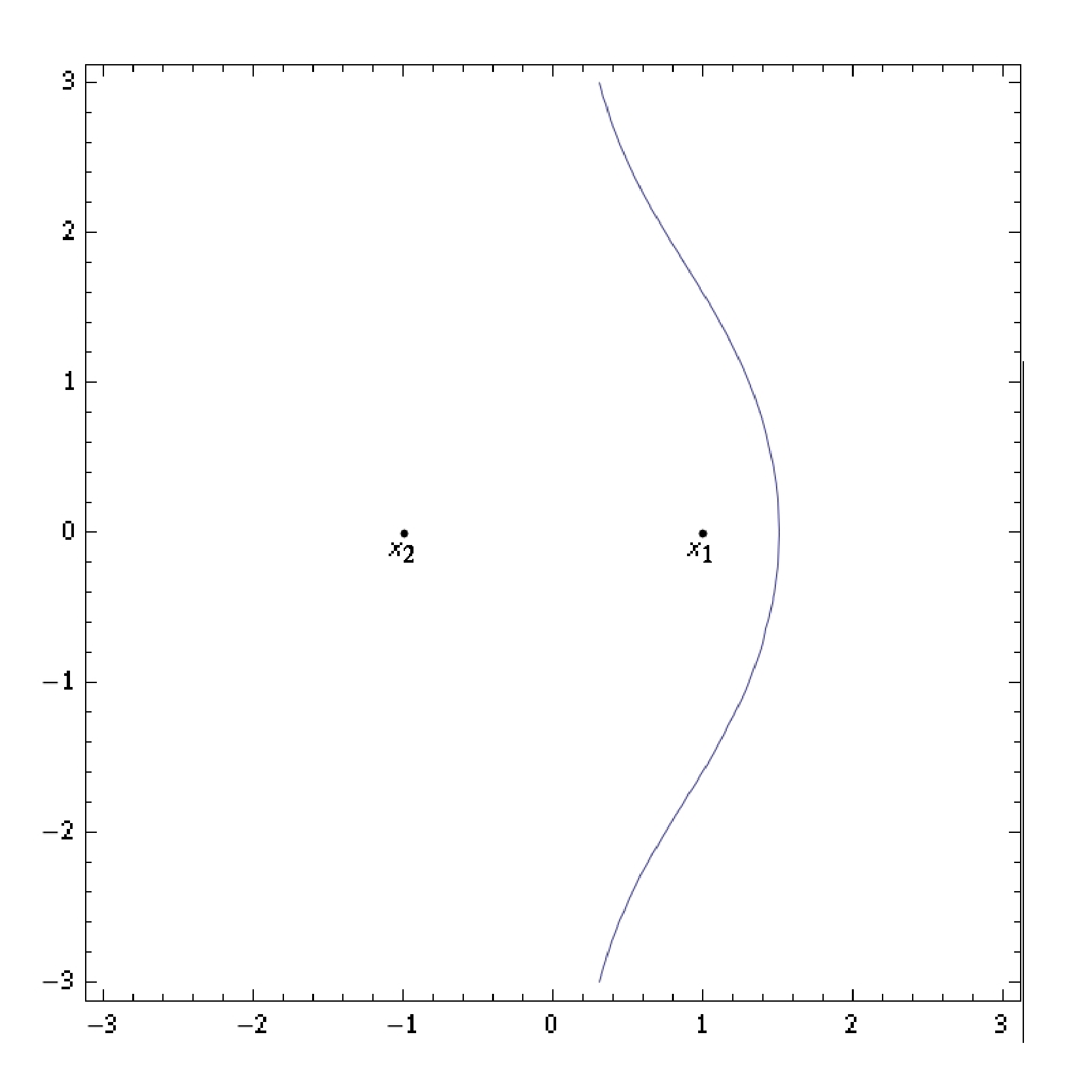}
\end{minipage}
$\dashrightarrow$
\begin{minipage}[c]{6cm}
 \includegraphics[scale=0.21]{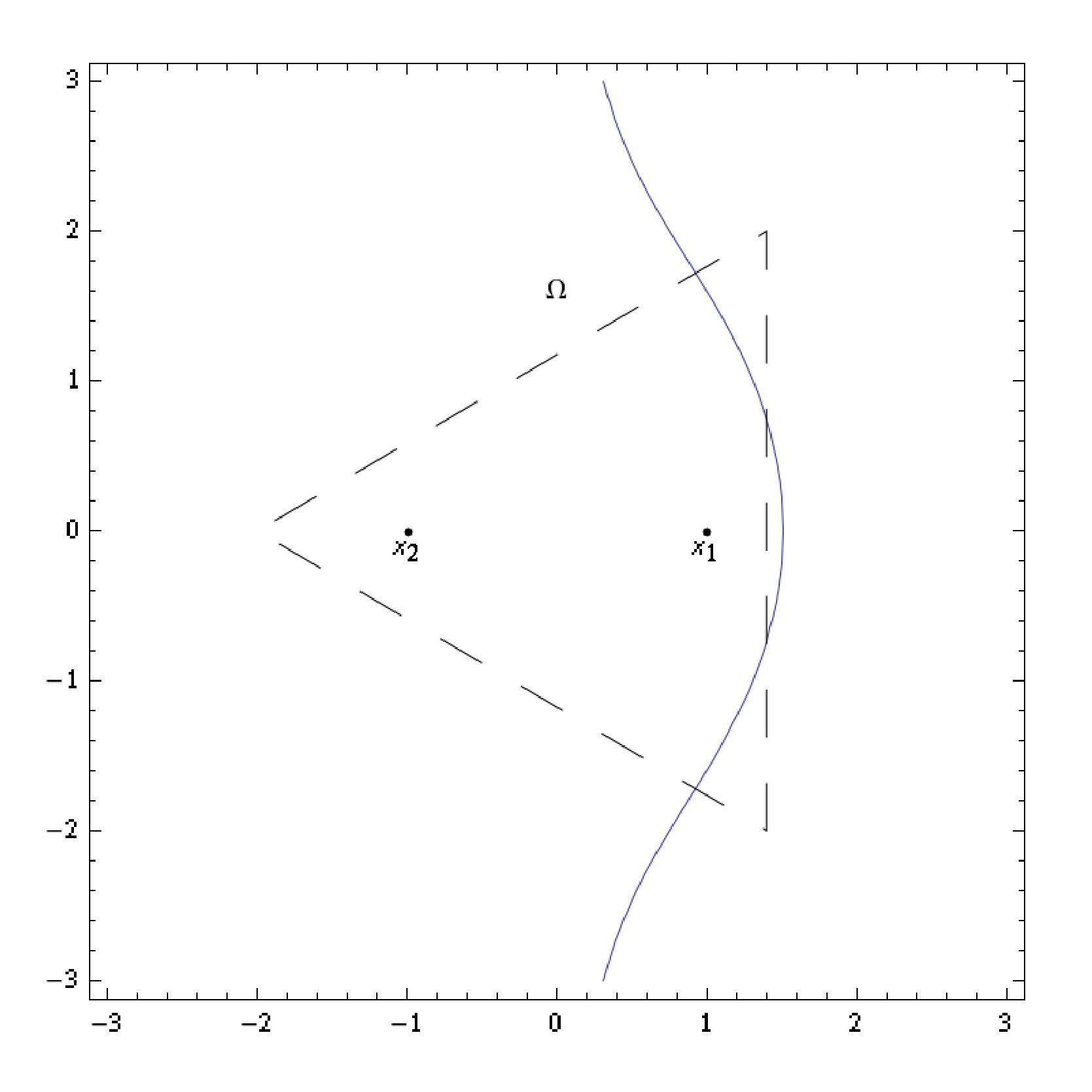}
\end{minipage}
\end{minipage}
\caption{$p=5$, $\alpha=100$}
\end{figure}

\begin{theorem}\label{euklid}
Let $n\geq 2$, $p\in (1,\infty)\backslash \{2\}$. Then, there is $\Omega\subset\mathbb{R}^n$ convex and compact, $x_1,x_2\in\Omega$, and $0<\lambda<1$ such that $T^{-1}(x_1)$ is not connected. $T$ is the optimal transportation map for $\mu$ and $\nu=\lambda \delta_{x_1}+(1-\lambda)\delta_{x_2}$ and cost function $c_p(x,y)$.
\end{theorem}
\begin{proof}
As the measure $\nu$ is a convex combination of just two Dirac-measures, we can safely work in $\mathbb{R}^2$. We will take $x_1=(1,0)$ and $x_2=(-1,0)$. In order to characterise the set $H^1_2$ one needs to study its boundary $\partial H^1_2$. However, the boundary is just some level set $\alpha$ of the function
\[
 m(x,y,p)=|x_1-(x,y)|^p-|x_2-(x,y)|^p=((x+1)^2-y^2)^{p/2}-((x-1)^2+y^2)^{p/2},
\]
where $(x,y)$ denotes some point in $\mathbb{R}^2$. Figure 1 and figure 2 show how these level sets look like for certain values of $p$ and $\alpha$. The figures show the typical behaviour of $m$ for $p<2$ and for $p>2$. They also show, why theorem \ref{euklid} holds.\\

In the case of $p<2$ (figure 1) we can draw a straight line from a point in $H^2_1$, that is on the left hand side of the level set, to a point in $H^1_2$ on the right hand side of the level set which intersects the level sets three times. Completing this line to a convex polygon $\Omega$ such that $x_1$ and $x_2$ are inside the polygon, we find a convex set $\Omega$ such that $H^1_2$ is not connected. To show that this is indeed possible it suffices to show that the graph of the level set has a minimum in $y=0$ and a change in sign in its second derivative for some $y>0$.\\

The case $p>2$ is slightly easier. The graph of the level set has a maximum at $y=0$. For $\alpha$ large enough, $x_1\in H^2_1$. Thus, we can take a segment of the hyperplane $x=c$ for some $c>1$ such that $(c,0)\in H^2_1$, which intersects with the graph of the level set two times, and complete this segment to a polygon analogously to the case above. For this, it suffices to show that the graph of the level set has a maximum in $y=0$.\\

By the implicit function theorem, if $\frac{\partial m}{\partial x}(x_0,y_0)\neq 0 $ then there is a function $\eta(y)$ such that in a neighbourhood of $(x_0,y_0)$ we have $m(\eta(y),y,p)=\alpha$. This condition holds for all $x>1$, the case we are interested in, and $1<p<\infty$. Moreover, we have that
\[
 \frac{\partial \eta}{\partial y}=-\frac{\partial m/\partial y}{\partial m/\partial x}(\eta(y),y)
\]
and for the second derivative of $\eta$ we get
\begin{eqnarray}
 \frac{\partial^2 \eta}{\partial y^2}=\frac{1}{\partial m /\partial x} \left[-\frac{\partial^2 m}{\partial x^2}\left(\frac{\partial \eta}{\partial y}\right)^2+2\frac{\partial^2 m}{\partial y\partial x}\frac{\partial \eta}{\partial y}-\frac{\partial^2 m}{\partial y^2}\right]
\end{eqnarray}
This means, that even though we do not know how $\eta$ explicitly looks like, we do know what its derivative at any point in $\mathbb{R}^2$ with $x>1$ is. And this is enough for our purpose. We can check, if $\eta$ has an extremum at $y=0$ and if we know that for any fixed $x>1$ the second derivative of $\eta$ changes its sign, we are done. Let us first check the extremal points of $\eta$.\\

Calculating the first derivative of $\eta$ yields:

 \[
\frac{p y \left((-1+x)^2+y^2\right)^{-1+\frac{p}{2}}-p y \left((1+x)^2+y^2\right)^{-1+\frac{p}{2}}}{-p (-1+x) \left((-1+x)^2+y^2\right)^{-1+\frac{p}{2}}+p
(1+x) \left((1+x)^2+y^2\right)^{-1+\frac{p}{2}}},
\]
which is zero at $y=0$. The second derivative at the point $(x,0)$ is:
\[
\frac{ \left((-1+x)^2\right)^{-1+\frac{p}{2}}- \left((1+x)^2\right)^{-1+\frac{p}{2}}}{- (-1+x) \left((-1+x)^2\right)^{-1+\frac{p}{2}}+ (1+x)
\left((1+x)^2\right)^{-1+\frac{p}{2}}}
\]
As we consider the case $x>1, p>1$ the denominator is strictly positive. As the numerator is negative for $p>2$ and positive for $p<2$, $\eta$ has at the point $(x,0)$ a maximum for $p>2$ and a minimum for $p<2$.\\
This settles the case for $p>2$. For $p<2$ it remains to show that the second derivative of $\eta$ has a change in sign for fixed $x$. The second derivative at the point $(x,y)$ is

\begin{eqnarray*}
&&\frac{p}{-p (-1+x)
\left((-1+x)^2+y^2\right)^{\frac{1}{2} (-2+p)}+p (1+x) \left((1+x)^2+y^2\right)^{\frac{1}{2} (-2+p)}}\times\\
 &\times& \Big[(-2+p) y^2 \left((-1+x)^2+y^2\right)^{\frac{1}{2} (-4+p)}+\left((-1+x)^2+y^2\right)^{\frac{1}{2} (-2+p)}\\
&-&(-2+p) y^2 \left((1+x)^2+y^2\right)^{\frac{1}{2}
(-4+p)}-\left((1+x)^2+y^2\right)^{\frac{1}{2} (-2+p)}\\
&+&\frac{p^2 y^2}{\left(p (-1+x) \left((-1+x)^2+y^2\right)^{\frac{1}{2}
(-2+p)}-p (1+x) \left((1+x)^2+y^2\right)^{\frac{1}{2} (-2+p)}\right)^2}\times\\
&\times& \Big(\left((-1+x)^2+y^2\right)^{\frac{1}{2} (-2+p)}-\left((1+x)^2+y^2\right)^{\frac{1}{2}
(-2+p)}\Big)^2 \times\\
&\times&\Big\{(-2+p) (-1+x)^2 \left((-1+x)^2+y^2\right)^{\frac{1}{2} (-4+p)}+\left((-1+x)^2+y^2\right)^{\frac{1}{2} (-2+p)}\\
&-&(-2+p) (1+x)^2
\left((1+x)^2+y^2\right)^{\frac{1}{2} (-4+p)}-\left((1+x)^2+y^2\right)^{\frac{1}{2} (-2+p)}\Big\}\\
&+&\frac{2 (-2+p) p y^2}{-p (-1+x) \left((-1+x)^2+y^2\right)^{\frac{1}{2} (-2+p)}+p (1+x) \left((1+x)^2+y^2\right)^{\frac{1}{2} (-2+p)}}\times\\
&\times& \left(-(-1+x) \left((-1+x)^2+y^2\right)^{\frac{1}{2}
(-4+p)}+(1+x) \left((1+x)^2+y^2\right)^{\frac{1}{2} (-4+p)}\right)\times\\
&\times& \left(-\left((-1+x)^2+y^2\right)^{\frac{1}{2} (-2+p)}+\left((1+x)^2+y^2\right)^{\frac{1}{2}
(-2+p)}\right)\Big]
\end{eqnarray*}

For $x>1$ the first fraction is positive. Thus, we can concentrate on the sum of the three terms inside the square brackets. The common denominator of these terms is a square. As we are interested in the sign, it suffices to look at the numerator after the expansion to the common denominator. This yields three terms, say $u,v$ and $w$. As we have already seen that the second derivative of $\eta$ is positive at $(x,0)$ (for $1<p<2$), it suffices to show that the numerator converges to zero from below as $y$ tends to infinity to be able to conclude with the mean value theorem that there is a change in sign (a null with change of sign). \\

The first term reads
\begin{eqnarray*}
u&=&\Big[(-2+p) y^2 \left((-1+x)^2+y^2\right)^{\frac{1}{2} (-4+p)}-(-2+p) y^2 \left((1+x)^2+y^2\right)^{\frac{1}{2}
(-4+p)}\\
&+&\left((-1+x)^2+y^2\right)^{\frac{1}{2} (-2+p)}-\left((1+x)^2+y^2\right)^{\frac{1}{2} (-2+p)}\Big]\times\\
&\times& \left(p (-1+x) \left((-1+x)^2+y^2\right)^{\frac{1}{2} (-2+p)}-p (1+x) \left((1+x)^2+y^2\right)^{\frac{1}{2}
(-2+p)}\right)^2
\end{eqnarray*}
By using functions like $g(z)=(x+1)^{\frac{p-4}{2}}$, we find the asymptotic $u\sim 8p^2(p-2)(3-p)xy^{3p-8}$. The second term reads
\begin{eqnarray*}
 v&=& p^2 y^2 \left(\left((-1+x)^2+y^2\right)^{\frac{1}{2} (-2+p)}-\left((1+x)^2+y^2\right)^{\frac{1}{2} (-2+p)}\right)^2\times\\
&\times& \Big[(p-2) (x-1)^2 \left((x-1)^2+y^2\right)^{\frac{1}{2}
(p-4)}-(p-2) (1+x)^2 \left((1+x)^2+y^2\right)^{\frac{1}{2} (p-4)}\\
&+&\left((-1+x)^2+y^2\right)^{\frac{1}{2} (-2+p)}-\left((1+x)^2+y^2\right)^{\frac{1}{2}(-2+p)}\Big],
\end{eqnarray*}
which behaves like $v\sim p^2(p-2)^3x^3 24 y^{3p-10}$. The last term is
\begin{eqnarray*}
w&=& 2 (-2+p) p y^2  \left(-\left((-1+x)^2+y^2\right)^{\frac{1}{2}
(-2+p)}+\left((1+x)^2+y^2\right)^{\frac{1}{2} (-2+p)}\right)\times\\
&\times&\left((1-x) \left((-1+x)^2+y^2\right)^{\frac{1}{2} (-4+p)}+(1+x) \left((1+x)^2+y^2\right)^{\frac{1}{2} (-4+p)}\right)\times \\
&\times& \left(-p (-1+x) \left((-1+x)^2+y^2\right)^{\frac{1}{2} (-2+p)}+p (1+x) \left((1+x)^2+y^2\right)^{\frac{1}{2}
(-2+p)}\right)
\end{eqnarray*}
This behaves asymptotically like $w\sim 16 p^2 (p-2)^2xy^{3p-8}$. In total, we have an asymptotic for $y$ tending to infinity like
\[
 u+v+w\sim 8(p-1)(p-2)p^2xy^{3p-8}\sim -y^{3p-8},
\]
as we assumed $1<p<2$. Thus, there is a change in sign of the second derivative of $\eta$ for any fixed $x>1$ proving the theorem.
\end{proof}

\subsection{Simply connected subsets of $\mathbb{R}^n$}
If we now replace the restriction of $\Omega$ being convex with the assumption of $\Omega$ being simply connected we are even worse off. Then, it can happen that $T^{-1}(x_i)$ has even infinitely many components. 
\begin{example}\label{infty}
Take two points in $\mathbb{R}^2$, say $x_1$ and $x_2$ from above, and for simplicity $c(x,y)=\frac{d^2}{2}(x,y)$ (this also works for other $p>1$). Arrange the constants $b_i$ such that $T(x_1)=x_2$. By induction, we now build  a set $\Omega$ such that $T^{-1}(x_1)$ has infinitely many components. The basic idea is as follows. $\partial H^1_2$ is a straight line. Everything on the left is transported to $x_2$ everything on the right to $x_1$. Just cut out a set with infinitely many components on the right hand side of $\partial H^1_2$:\\
Cut a nice set $\Omega_0$ out of $H^2_1$ such that some segment of the straight line $\partial H^2_1$ is a connected subset of the boundary of $\Omega_0$, say $G=\partial H^2_1\cap \partial\Omega_0$. Then, take a nice subset $W_1$ of $H^1_2$ such that $\partial W_1\cap \partial H^1_2 = g_1\subset G$. Set $\Omega_1=\Omega_0 \cup W_1$. \\
Let $W_2$ be a version of $W_1$ scaled down by some factor $\kappa/2$ and translated such that $g_2=\partial W_2\cap \partial H^1_2 \subset G$ and $\inf_{x\in W_2,y\in W_1} d(x,y)\geq \kappa/2$. Then set $\Omega_2=\Omega_1\cup W_2$.\\
Let $\Omega_n$ be constructed and let $W_{n+1}$ be a version of $W_n$ scaled down by the factor $\kappa/2$ and translated such that $g_{n+1}=\partial W_{n+1}\cap \partial H^1_2 \subset G$, $\inf_{x\in W_{n+1},y\in W_n} d(x,y)\geq (\kappa/2)^n$ and $\inf_{x\in W_{n+1},y\in W_j} d(x,y)\geq \sum_{k=j}^n2(\kappa/2)^k$. Then set $\Omega_{n+1}=\Omega_n\cup W_{n+1}$.\\
Following this procedure inductively, we set $\Omega=\Omega_\infty$ which will look like figure 3

\begin{figure}
 \includegraphics[scale=0.2]{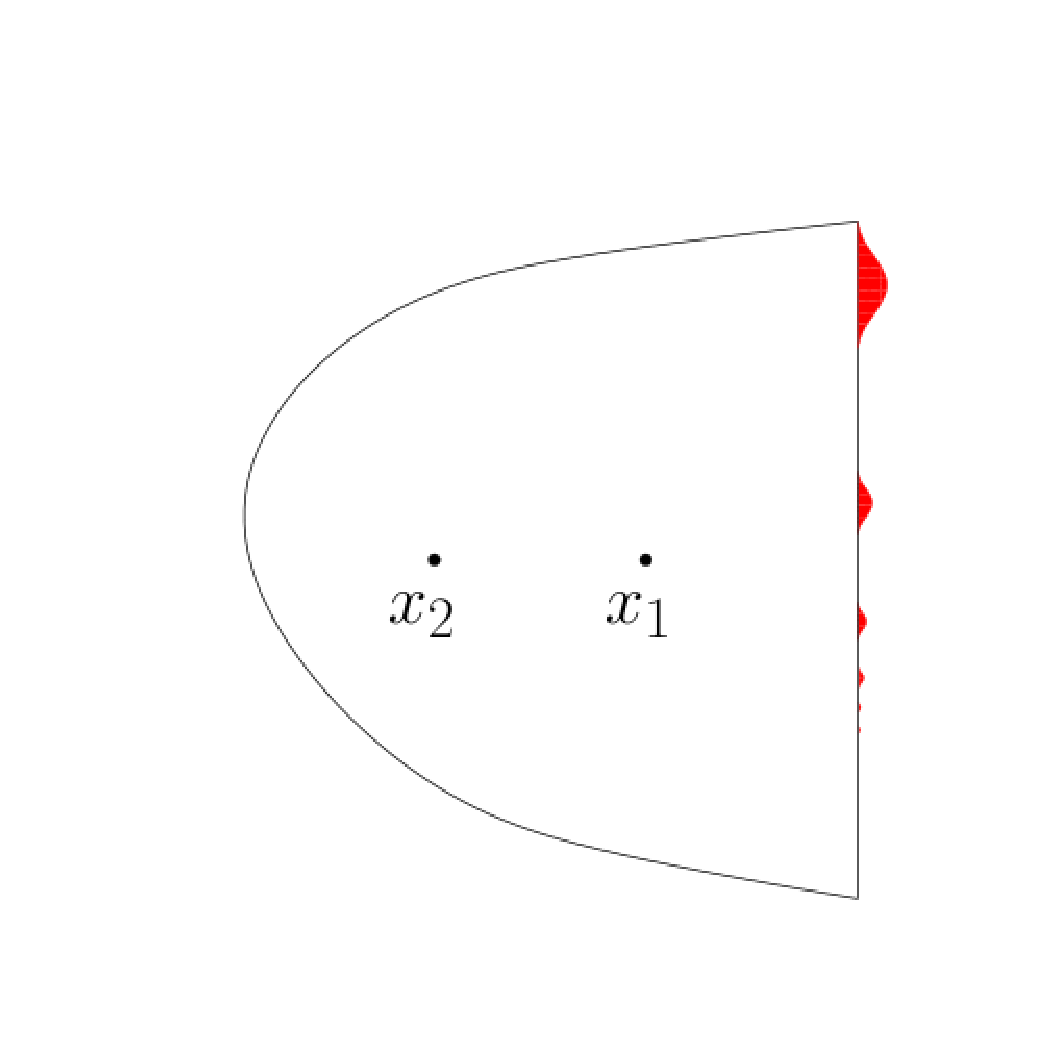}
\caption{Example \ref{infty}}
\end{figure}

The red area is transported to $x_1$, disconnected and has infinitely many components. 
\end{example}

\section{General manifolds and ``$p=1$''}\label{general manifolds}
We now want to consider the problem on a general complete compact Riemannian manifold. For the case $p=1$ the situation is very nice for all manifolds and we get a good geometric description of the cells of the Johnson Mehl diagram. Due to the triangle inequality, we have the following
\begin{theorem}
Let $M$ be a compact connected complete Riemannian manifold without boundary with $dim(M)\geq 2$. Consider the optimal transport problem from $\mu$ to $\nu$ with cost function $c_1(x,y)=d(x,y)$, the geodesic distance. Then, for any $i$ with $\lambda_i>0$ the set $T^{-1}(x_i)$ is simply connected and starlike with respect to $x_i$.
\end{theorem}

\begin{remark}
Note that in this case the transport map is unique even though we are working with the cost function $c_1(x,y)=d(x,y)$. The reason is that we transport the mass to a discrete measure. The ambiguous sets $\partial H^i_j=\{x\in M: \Phi_i(x)=\Phi_j(x)\}$ are null sets. Thus, their (countable) union is a null set and for any $z\in M\backslash \bigcup_{i,j}\partial H^i_j$ there is a unique transport plan.
\end{remark}

\begin{proof}
We claim that the halfspace $H^i_j$ is starlike with respect to $x_i$. Then, if $z\in  T^{-1}(x_i)=\bigcap_{j\neq i} H^i_j$ we have $z\in H^i_j$ for all $j$. As $H^i_j$ is starlike with respect to $x_i$, the geodesic $\gamma_t$ connecting $z$ and $x_i$ lies entirely inside $H^i_j$. And this holds for all $j$. Thus, $\{\gamma_t \}\subset \bigcap_{j\neq i} H^i_j$ and $T^{-1}(x_i)$ is not only connected but also starlike with respect to $x_i$. Hence, it suffices to prove the claim. \\
Assume $x_i\in\partial H^i_j$ and w.l.o.g. $b_i=0$. Then, we have
\[
 \Phi_i(x_i)=0=\Phi_j(x_i) \Rightarrow b_j=d(x_j,x_i).
\]
The set $N=\{z\in M: d(x_j,z)=d(x_j,x_i)+d(x_i,z)\}$ is a $\mu$-null set. For all $z\notin N$ we have
\[
 \Phi_j(z)=-d(x_j,z)+b_j>-d(x_j,x_i)+b_j-d(x_i,z)=\Phi_i(z)
\]
This implies that $T^{-1}(x_i)=\emptyset$ contradicting the assumption of $\lambda_i>0$. Thus, $x_i\notin\partial H^i_j$

Assume $T(x_i)\neq x_i$. Then, there is a $j\neq i$ such that $T(x_i)=x_j$, i.e. $\Phi_j(x_i)=-d(x_i,x_j)+b_j>b_i=\Phi_i(x_i)$. Then, we have for any $q\in M$
\[
 -d(q,x_j)+b_j\geq -d(q,x_i)-d(x_i,x_j)+b_j>-d(q,x_i)+b_i.
\]
This implies, that $T^{-1}(x_i)=\emptyset$ contradicting the assumption of $\lambda_i>0$. Thus, $T(x_i)=x_i$. 

Take any $w\in T^{-1}(x_i)$ (hence, $\Phi_i(w)>\Phi_j(w)$ for all $j\neq i$) and $q\in M$ such that $d(x_i,w)=d(x_i,q)+d(q,w)$, i.e. $q$ lies on the minimising geodesic from $x_i$ to $w$. Then, we have for any $j\neq i$ by using the triangle inequality once more
\begin{eqnarray*}
 -d(q,x_i)+b_i&=&-d(x_i,w)+d(q,w)+b_i\\
&\geq& -d(x_i,w)+b_i+d(w,x_j)-d(q,x_j)\\
&>&-d(q,x_j)+b_j,
\end{eqnarray*}
which means that $\Phi_i(q)>\Phi_j(q)$ for all $j\neq i$. Hence, $q\in \bigcap_{j\neq i} H^i_j$ proving the claim.
\end{proof}

However, as mentioned above, already for the case $p=2$ on a compact subset of the hyperbolic space the cells $T^{-1}(x_i)$ can be disconnected (see Example 3.8 in \cite{kim-2007}). Results of various simulations hint at this not being special for $p=2$. We believe, that for the hyperbolic space, this holds for all $p>1$. In the case of a compact manifold with positive and negative curvature, we can be even worse off.

\begin{example}\label{kreiswiggle}
Let $p>1$ be given. Consider two points on the sphere which are the antipodes of each other, say $x$ and $x'$. Fix the transport map $T=\exp (\nabla \Phi)$ with $\Phi=\max\{\Phi_1,\Phi_2\}$ and $\Phi_1(y)=-\frac{1}{p} d^p(x,y)+b_1$ and $\Phi_2(y)=-\frac{1}{p} d^p(x',y)$, as above. Assume, that $b_1$ is such that $\Phi_1(x')=\Phi_2(x')$. In particular, we have $T(S^n)=x$. We now want to deform the sphere into a manifold $M$ with positive and negative curvature while fixing the transport map $T$ (the metric of the manifold will change and therefore we will end up with another distance function, but we will suppress that in the notation).\\

Fix a geodesic connecting $x$ and $x'$ (that is one half of a great circle connecting $x$ and $x'$). The $x'$-equator divides the geodesic into two segments. One of the segments, say $G$, has the property that all of its points are closer to $x'$ than to $x$. Let us fix a point $g\in G$ not equal to $x'$. By adding a very thin spike to $S^n$ around $g$ and making it, if necessary, very long, we can achieve, that the top of the spike is transported to $x'$. The reason is, that the cost function $c(z,y)=\frac{1}{p}d^p(z,y)$ is strictly convex and the points on the spike are closer to $x'$ than to $x$ (take $\kappa$ arbitrary, $p>1$, then, there is $R$ such that for all $r>R$ we have $-(r+1)^p+\kappa < -(r-1)^p$). Then, in a similar manner to example \ref{infty} we can add a sequence of spikes not touching each other, getting smaller and smaller and converging to a ``zero-spike'' at $x'$ (see figure 4). In order that this yields infinitely many components of the cell $T^{-1}(x')$ we need the assumption $\Phi_1(x')=\Phi_2(x')$.
\begin{figure}
 \includegraphics[scale=0.2]{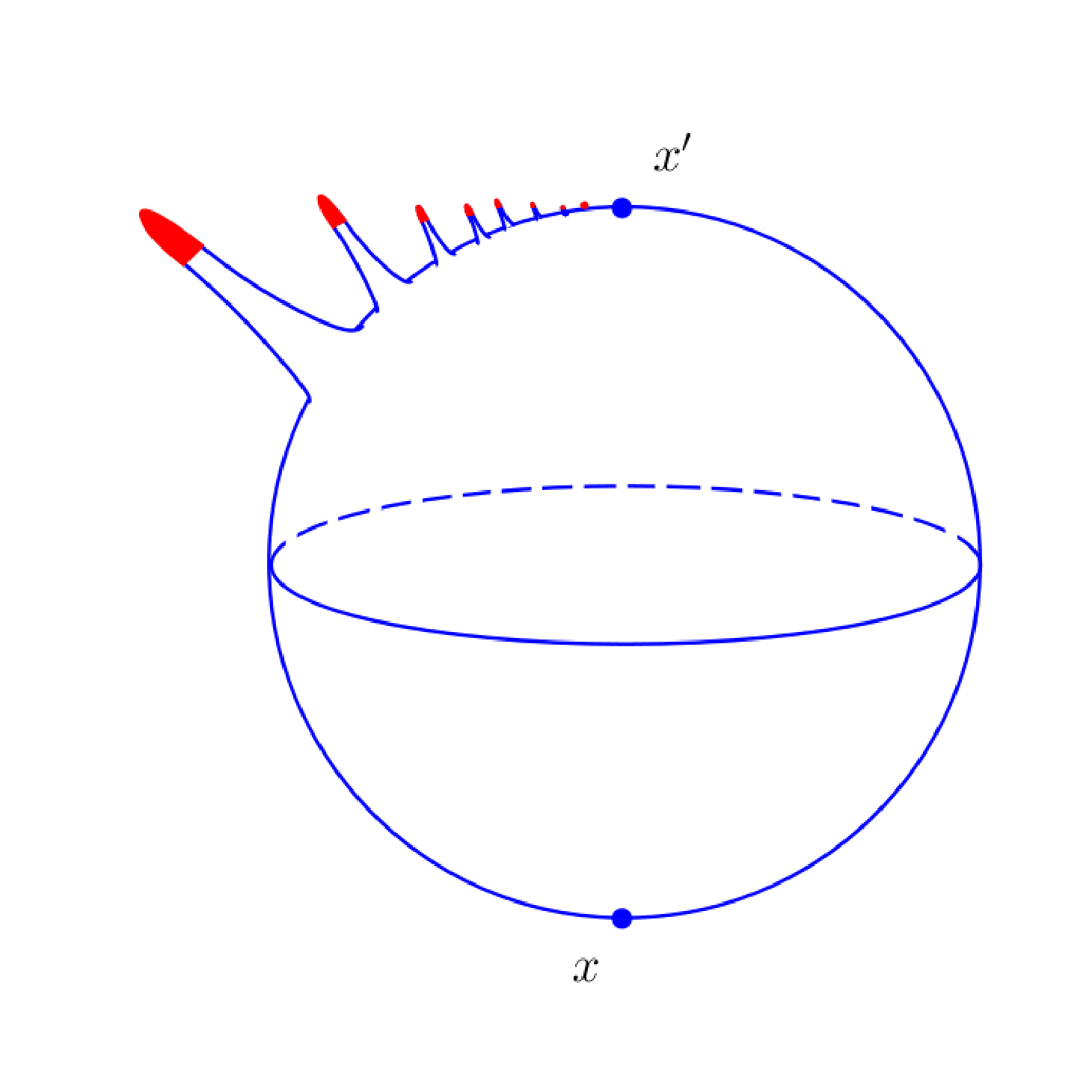}
\caption{Example \ref{kreiswiggle}}
\end{figure}

\end{example}


\begin{thebibliography}{MTW05}

\bibitem[AHA92]{142747}
Franz Aurenhammer, Friedrich Hoffman, and Boris Aronov.
\newblock Minkowski-type theorems and least-squares partitioning.
\newblock In {\em SCG '92: Proceedings of the eighth annual symposium on
  Computational geometry}, pages 350--357, New York, NY, USA, 1992. ACM.

\bibitem[GM96]{GangboMcCann1996}
Wilfrid Gangbo and Robert~J. McCann.
\newblock The geometry of optimal transportation.
\newblock {\em Acta Math.}, 177(2):113--161, 1996.

\bibitem[KM07]{kim-2007}
Young-Heon Kim and Robert~J. McCann.
\newblock Continuity, curvature, and the general covariance of optimal
  transportation, 2007.

\bibitem[Loe09]{loeper2009regularity}
G.~Loeper.
\newblock {On the regularity of solutions of optimal transportation problems}.
\newblock {\em Acta Mathematica}, 202(2):241--283, 2009.

\bibitem[LZ08]{lautensack2007}
Claudia Lautensack and Sergei Zuyev.
\newblock Random {L}aguerre tessellations.
\newblock {\em Adv. in Appl. Probab.}, 40(3):630--650, 2008.

\bibitem[McC01]{mccann01polar}
Robert~J. McCann.
\newblock Polar factorization of maps on {R}iemannian manifolds.
\newblock {\em Geom. Funct. Anal.}, 11(3):589--608, 2001.

\bibitem[MTW05]{ma2005}
Xi-Nan Ma, Neil~S. Trudinger, and Xu-Jia Wang.
\newblock Regularity of potential functions of the optimal transportation
  problem.
\newblock {\em Arch. Ration. Mech. Anal.}, 177(2):151--183, 2005.

\bibitem[RU97]{RueUckelmann1997}
L.~R\"uschendorf and L.~Uckelmann.
\newblock On optimal multivariate couplings.
\newblock 1997.

\bibitem[Stu09]{sturm-2009}
K.T. Sturm.
\newblock {Entropic Measure on Multidimensional Spaces}.
\newblock {\em eprint arXiv: 0901.1815}, 2009.

\bibitem[Vil09]{villani2-2009}
C\'edric Villani.
\newblock {\em Optimal transport. Old and new.}
\newblock Grundlehren der Mathematischen Wissenschaften 338. Berlin: Springer.,
  2009.

\end{thebibliography}
\end{document}